\input ./style/arxiv-vmsta.cfg
\documentclass[numbers,compress,v1.0.1]{vmsta}
\usepackage{vtexbibtags}

\volume{6}% Updated by VTEXPTS2LaTeX.exe, 22.11.2019 09:28
\issue{4}% Updated by VTEXPTS2LaTeX.exe, 22.11.2019 09:28
\pubyear{2019}
\firstpage{419}% Updated by VTEXPTS2LaTeX.exe, 22.11.2019 09:28
\lastpage{441}% Updated by VTEXPTS2LaTeX.exe, 22.11.2019 09:28
\aid{VMSTA142}% Updated by VTEXPTS2LaTeX.exe, 23.09.2019 11:20
\doi{10.15559/19-VMSTA142}% Updated by VTEXPTS2LaTeX.exe, 23.09.2019 11:20
\articletype{research-article}

\startlocaldefs
\newtheorem{thm}{Theorem}
\newtheorem{lemma}{Lemma}
\theoremstyle{definition}
\newtheorem{remark}{Remark}
\allowdisplaybreaks
\ifdefined\HCode % tex4ht
\def\index#1{}
\else % LaTeX
\fi
\endlocaldefs

\begin{document}

\begin{frontmatter}
\pretitle{Research Article}

\title{Logarithmic L\'{e}vy process directed by Poisson subordinator}

\author{\inits{P.}\fnms{Penka}~\snm{Mayster}\thanksref{cor1}\ead[label=e1]{penka.mayster@math.bas.bg}}%\thanksref{f1}
\author{\inits{A.}\fnms{Assen}~\snm{Tchorbadjieff}\ead[label=e2]{assen@math.bas.bg}}%\thanksref{f1}
\thankstext[type=corresp,id=cor1]{Corresponding author.}
\address{Institute of Mathematics and Informatics,
\institution{Bulgarian Academy of Sciences},
Acad.~G.~Bonchev~street,~Bloc~8, 1113~Sofia,~\cny{Bulgaria}}
%\address[]{\institution{}, ..., \cny{}}

%\thankstext[id=f1]{}

%\dedicated{}

%\markboth{Authors}{Title}
\markboth{P. Mayster, A. Tchorbadjieff}{Logarithmic L\'{e}vy process directed by Poisson subordinator}%{Logarithmic L\'{e}vy process}

\begin{abstract}
Let $ \{L(t), t\geq 0 \}$ be a L\'{e}vy process with representative
random variable $ L(1)$ defined by the infinitely divisible logarithmic
series distribution. We study here the transition probability and
L\'{e}vy measure of this process. We also define two subordinated
processes. The first one, $Y(t)$, is a Negative-Binomial process
$X(t)$ directed by Gamma process. The second process, $Z(t)$, is a
Logarithmic L\'{e}vy process $L(t)$ directed by Poisson process. For
them, we prove that the Bernstein functions of the processes
$ L(t)$ and $ Y(t)$ contain the iterated logarithmic function. In
addition, the L\'{e}vy measure of the subordinated process $ Z(t)$ is a
shifted L\'{e}vy measure of the Negative-Binomial process $ X(t)$. We
compare the properties of these processes, knowing that the total masses
of corresponding L\'{e}vy measures are equal.
\end{abstract}

\received{\sday{5} \smonth{4} \syear{2019}}% Updated by VTEXPTS2LaTeX.exe, 23.09.2019 11:20
\revised{\sday{13} \smonth{9} \syear{2019}}% Updated by VTEXPTS2LaTeX.exe, 23.09.2019 11:20
\accepted{\sday{13} \smonth{9} \syear{2019}}% Updated by VTEXPTS2LaTeX.exe, 23.09.2019 11:20
\publishedonline{\sday{4} \smonth{10} \syear{2019}}

\end{frontmatter}

%s1 #&#
\section{Introduction}%
\label{sec1}
Let $ \{L(t), t\geq 0 \}$ be a L\'{e}vy process\index{L\'{e}vy process} with representative
random variable $ L(1)$ defined by the infinitely divisible logarithmic
series\index{infinitely divisible logarithmic series} distribution. The distribution of any L\'{e}vy process\index{L\'{e}vy process} is
completely determined by the distribution of its representative random
variable, which is infinitely divisible \cite{Ste}. The
probability generating function (p.g.f.) of the random variable
$ L(1)$ is expressed by the Gauss hypergeometric function
$\mbox{}_{2}F_{1}(1,1;2;z)$ \cite{JKK}. This makes the usage of
enumerative combinatorics methods indispensable in this study
\cite{Pit}. Thus, using the partial Bell polynomials\index{partial Bell polynomials} we obtain an
explicit representation of the transition probability\index{transition probability} and L\'{e}vy
measure\index{L\'{e}vy measure} of this process. But, first of all, we distinguish two
logarithmic distributions.

The logarithmic series\index{logarithmic series} distribution supported by positive integers
$ N=\{1,2, \ldots\}$ was firstly introduced by R.A. Fisher, A.S. Corbet  and C.B. Williams (1943) \cite{Fisher}. It is not an
infinitely divisible distribution. It is a L\'{e}vy measure\index{L\'{e}vy measure} for the well-known Negative-Binomial process. %They say, that
The paper
\cite{Fisher} represents an impressive combination of empirical data and
mathematical analysis, remaining a model for ecology today.

The logarithmic series\index{logarithmic series} distribution supported by nonnegative integers
$ Z_{+}=\{0,1,\allowbreak \ldots\}$ is a particular case of the Kemp generalized
hypergeometric probability distribution (1956) \cite{Kemp}. Its
infinite divisibility was proved by K. Katti (1967)
\cite{Katti}. Infinitely divisible random variables with values in
$Z_{+}$ were first studied by Feller (1968) \cite{Fel}, where it
is shown that on $Z_{+}$ the infinitely divisible distributions\index{infinitely divisible distributions} coincide
with the compound Poisson distributions. A historical review on the
origin of infinitely divisible distributions\index{infinitely divisible distributions} ``from de Finetti's problem
to L\'{e}vy--Khintchine formula'' is presented by F. Mainardi  and
S. Rogosin (2006) \cite{MR}.

At the present time, there are many works related to the topics of
infinite divisibility and discrete distributions. Some of them are
monographs of F.W. Steutel and K. Van Harn \cite{Ste}, and
N.L. Johnson, A.W. Kemp and S. Kotz \cite{JKK}.
Integer-valued L\'{e}vy processes\index{L\'{e}vy process} and their application in financial
econometrics are developed by O.E.~Barndorff-Nielsen, D. Pollard and
N. Shephard \cite{Barn}. The compositions of Poisson and Gamma
processes are investigated by K. Buchak and L. Sakhno in
\cite{BuSa,BS}. The consecutive subordinations\index{consecutive subordinations} of Poisson and Gamma
processes realized on two sequences containing only four new processes
are studied in \cite{Ma}. It is shown there how the additional
randomness, caused by random time, is accumulated. The transformation
of the Poisson random measure and the jump-structure of the subordinated
process\index{subordinated process} is described in \cite{M}. Other interesting integer-valued
Markov processes are derived from Markov branching processes. In the
model of branching particle system with a random initial condition, we
obtained distributions describing the number of particles at time
$t$, corresponding to the compound Poisson processes over the radius of
a flux of particles \cite{MT}.

The subordination\index{subordination} by Bochner is also developed in many books and
articles related to applications in financial mathematics and functional
analysis, see \cite{Sato,Ci}. There is a special Chapter 3 in
\cite{B} devoted to the subordinators. The L\'{e}vy measure\index{L\'{e}vy measure} and
potential kernel are also considered in \cite{Berg}. The
properties of the Bernstein functions\index{Bernstein function} are studied in \cite{Sch}.
The theory of subordinators and inverse subordinators is applied to
study risk processes in insurance models by N. Leonenko et al. in
\cite{KLP,LST}.

Our work in the topic is described in the following
%six paragraphs.
sections. In
Section~\ref{sec2} we introduce the infinitely divisible logarithmic series\index{infinitely divisible logarithmic series}
distribution and its $m$-fold convolution. Our main tools of
investigation are the Gauss hypergeometric function and partial Bell
polynomials,\index{partial Bell polynomials} Stirling numbers\index{Stirling numbers} and harmonic numbers. In Section~\ref{sec3} we
present two methods defining the transition probability\index{transition probability} of the L\'{e}vy
process $ L(t)$ -- starting with the L\'{e}vy measure\index{L\'{e}vy measure} or starting with
p.g.f. $ F(t,s) = E[s^{L(t)}]$ and its Taylor expansion. Then, in the
following two Sections~\ref{sec4} and \ref{sec5} we consider the subordinated processes\index{subordinated process}
$ Y(t)$ and $Z(t)$. They are obtained respectively from
the Negative-Binomial process $X(t)$ directed by the Gamma one and the Logarithmic
L\'{e}vy process\index{L\'{e}vy process} $ L(t)$ directed by the Poisson one. In this study of
subordinated processes,\index{subordinated process} we proceed also by two methods -- either
integrating the transition probability\index{transition probability} of the ground process,\index{ground process} as it is
shown in \cite{Sato}, or constructing the compound Poisson process
with a prior defined L\'{e}vy measure.\index{L\'{e}vy measure} The Bernstein functions\index{Bernstein function} of the
processes $ L(t)$ and $ Y(t)$ contain the iterated logarithmic function.
The L\'{e}vy measure\index{L\'{e}vy measure} of $ Z(t)$ is a shifted L\'{e}vy measure\index{L\'{e}vy measure} of
$ X(t)$. We compare the behaviour of all these processes in order to
understand better the place of the Logarithmic L\'{e}vy process\index{L\'{e}vy process} in the
picture of compound Poisson processes. Several combinatorics identities
arrive as auxiliary results. Finally, some applications derived from
studied processes are explained and demonstrated in Section~\ref{sec6}.

%s2 #&#
\section{Infinitely divisible logarithmic series\index{infinitely divisible logarithmic series} distribution and the Gauss hypergeometric function}
\label{sec2}
The Gauss hypergeometric function is defined in \cite{JKK}, page
20, for $ |z|<1$ by
\begin{equation*}
\mbox{}_{2}F_{1}(c,d;g;z)=\sum ^{\infty }_{k=0} \frac{[c]_{k \uparrow }[d]_{k
\uparrow }}{[g]_{k \uparrow }} \frac{z^{k}}{k!},
\end{equation*}
where the increasing factorial, known as Pochhammer's symbol, is denoted
as $[c]_{k \uparrow }=c(c+1)\cdots (c+k-1)$, $[c]_{0 \uparrow }=1$. In
particular,
\begin{equation*}
\mbox{}_{2}F_{1}(1,1;2;z)=\sum ^{\infty }_{k=0} \frac{k!k!}{(k+1)!} \frac{z
^{k}}{k!}=\frac{-\log (1-z)}{z}.
\end{equation*}
By definition, given in \cite{Ste}, Chapter 2, Example 11.7, the
infinitely divisible random variable $ L(1)$ with logarithmic series\index{logarithmic series}
distribution has the probability mass function (p.m.f.) supported by
$ \{0, 1, 2,\ldots\}$ and given as follows:
%
%e1 #&#
\begin{equation}
\label{LL}
P(L(1)=k)=\frac{1}{A} \frac{\alpha ^{k+1}}{(k+1)}, \quad 0<\alpha <1, A=-
\log (1-\alpha ), k=0, 1, \ldots .
\end{equation}
The p.g.f. defined by $F(1,s)=E[s^{L(1)}]$, $|s |\leq 1$, is
\begin{equation*}
F(1,s)=\sum ^{\infty }_{k=0}\frac{\alpha ^{k+1} }{k+1} \frac{s^{k}}{A}=\frac{
\alpha }{A}\sum ^{\infty }_{k=0}\frac{(\alpha s)^{k}}{k+1}= \frac{-
\log (1-\alpha s)}{As}
\end{equation*}
and can be presented as follows:
%
%e2 #&#
\begin{equation}
\label{G}
F(1,s)=\frac{\alpha }{A} (1+G(s)),\quad G(s)=\sum ^{\infty }_{k=1}\frac{(
\alpha s)^{k}}{k+1},\quad G(1)=\frac{A-\alpha }{\alpha }.
\end{equation}
We remark that the following simple representation is a starting point
of the Taylor expansion,
\begin{equation*}
\mbox{}_{2}F_{1}(1,1;2;\alpha s)= 1+G(s).
\end{equation*}
Let us denote the finite sum of random variables $L_{1}, L_{2},\ldots, L
_{m} $, being independent copies of $L(1)$, as
%
%e3 #&#
\begin{equation}
\label{L}
L(m)=\sum ^{m}_{j=1}L_{j}.
\end{equation}
By convolution of p.m.f. we express the probability distribution of the
random variable $ L(2)$ by means of the harmonic numbers as follows:
\begin{equation*}
P(L(2)=n )=\left (\frac{\alpha }{A}\right )^{2}
\frac{2\alpha ^{n}}{n+2}\left (1+\frac{1}{2}+\frac{1}{3}+\cdots+
\frac{1}{n+1}\right ),\quad n=0, 1, 2,\ldots.
\end{equation*}
Knowing the infinite divisibility of $L(1)$ we write the p.g.f. of
$L(m)$ (\ref{L}) by the power, $ F(m,s)=\{F(1,s)\}^{m} $, namely:
\begin{equation*}
\left (\frac{-\log (1-\alpha s)}{A s}\right )^{m}=\left (\frac{\alpha
}{A} (1+G(s))\right )^{m} .
\end{equation*}
We present here two methods on expanding $ F(m,s)$ as power series of
$s$, expanding only $(-\log (1-\alpha s))^{m} $, or with the binomial
expanding of $(\frac{\alpha }{A} (1+G(s)))^{m}$. For this purpose, the
function $ G(s)$ is presented as an exponential generating function:\index{exponential generating function}
%
%e4 #&#
\begin{equation}
\label{GE}
G(s)=\sum ^{\infty }_{k=1} g_{k} \frac{s^{k}}{k!}, \quad g_{k}= \frac{
\alpha ^{k} k!}{k+1}.
\end{equation}
For convenience, we denote the sequences by bullet, as it is shown in
\cite{Pit},
\begin{equation*}
g_{\bullet }= (g_{1}, g_{2}, \ldots).
\end{equation*}
Similarly, the sequences of the powers are expressed by $a^{\bullet }=
(a, a^{2},\ldots)$, and in particular: $1^{\bullet }= (1,1, \ldots)$. In both
cases of expanding we use the Faa di Bruno formula and the partial Bell
polynomials $B_{n,k}$\index{partial Bell polynomials} \cite{Pit}, allowing to express the power
$ [G(s)]^{k}$ as follows:
%
%e5 #&#
\begin{equation}
\label{B}
\frac{(G(s))^{k} }{k!} =\frac{1}{k!} \left ( \sum ^{\infty }_{j=1}g
_{j} \frac{s^{j}}{j!}\right )^{k} = \sum ^{\infty }_{n=k} B_{n,k}(g
_{\bullet }) \frac{s^{n}}{n!},
\end{equation}
where
\begin{equation*}
B_{n,k}(g_{\bullet })= \sum _{(k_{1},k_{2},\cdots, k_{n})} \frac{n!g_{1}
^{k_{1}}\ldots g_{n}^{k_{n}}}{k_{1}!(1!)^{k_{1}}\cdots k_{n}!(n!)^{k_{n}}}.
\end{equation*}
The sum is over all partitions of $n$ into $k$ parts, that is over all
nonnegative integer solutions $(k_{1}, k_{2},\ldots, k_{n})$ of the
equations:
\begin{equation*}
k_{1}+2k_{2}+\cdots+nk_{n}=n,\quad k_{1}+k_{2}+\cdots+k_{n}=k.
\end{equation*}
For example, $ B_{n,1}(x_{\bullet })=x_{n}$, $ B_{n,n}(x_{\bullet })=(x
_{1})^{n}$ and
%
%e6 #&#
\begin{equation}
\label{Bn}
B_{n,k}(a^{\bullet }b x_{\bullet })=a^{n} b^{k} B_{n,k}(x_{\bullet }).
\end{equation}
The falling factorials are defined as follows:
\begin{equation*}
[x]_{n \downarrow }=x(x-1)\cdots (x-n+1)=
\frac{\Gamma (x+1)}{\Gamma (x+1-n)}.
\end{equation*}
Let us denote the Stirling numbers\index{Stirling numbers} of the \textbf{first} kind by
$|s(n,k)|$ and $ s(n,k)$, respectively, -- unsigned
\begin{equation*}
|s(n,k)|:= B_{n,k}((\bullet -1)!)=B_{n,k}(0!,1!,2!,\ldots.),
\end{equation*}
and signed $s(n,k)$ depending on the parity of $n-k$ given by
\begin{equation*}
s(n,k)=(-1)^{n-k}|s(n,k)|.
\end{equation*}
The Stirling numbers\index{Stirling numbers} of the first kind transform the factorials into powers,
%
%e7 #&#
\begin{equation}
\label{sn}
[x]_{n \uparrow }=\sum ^{n}_{k=0}|s(n,k)|x^{k} ,\quad [x]_{n \downarrow
}=\sum ^{n}_{k=0}s(n,k)x^{k} .
\end{equation}
Thus, having these definitions and relations, the following useful lemma
is formulated and proved.
%
%l1 #&#
\begin{lemma}
The m-fold convolution of the infinitely divisible logarithmic series\index{infinitely divisible logarithmic series}
distribution (\ref{LL}) can be equivalently expressed in the following
forms:
%
%e8 #&#
\begin{equation}
\label{Lm}
P(L(m)=n )=\left (\frac{\alpha }{A}\right )^{m}
\frac{\alpha ^{n} }{n!}|s(m+n,m)| \frac{m! n!}{(m+n)!},
\quad n=0,1,2,\ldots,
\end{equation}
or % equivalently,
%
%e9 #&#
\begin{equation}
\label{Lmn}
P(L(m)=n )=\left (\frac{\alpha }{A}\right )^{m}
\frac{\alpha ^{n} }{n!} \sum ^{m\wedge n}_{k=0}\frac{m!}{(m-k)!} B_{n,k}(c
_{\bullet }),\quad B_{0,0}=1,
\end{equation}
where
\begin{equation*}
c_{\bullet }=\left (\frac{ k!}{k+1}, k=1,2,\ldots\right ),\quad B_{n,0}=B
_{0,k}=0,\ n>0,\ k>0.
\end{equation*}
\end{lemma}
\begin{proof}
Expanding only the logarithmic function $(-\log (1-\alpha s)) $ we obtain
\begin{equation*}
-\log (1-\alpha s)=\sum ^{\infty }_{k=1}\frac{(\alpha s)^{k}}{k}=
\sum ^{\infty }_{k=1}\frac{(k-1)!(\alpha s)^{k}}{k!}.
\end{equation*}
Then, for the powers we have
\begin{equation*}
\frac{1}{m!}\left (-\log (1-\alpha s)\right )^{m}=\sum ^{\infty }_{k=m}
B_{k,m}( \alpha ^{\bullet }(\bullet -1)!) \frac{s^{k}}{k!}
\end{equation*}
and
\begin{equation*}
\frac{1}{m!}\left (\frac{-\log (1-\alpha s)}{A s}\right )^{m}=\frac{1}{A
^{m}}\sum ^{\infty }_{k=m} \alpha ^{k} |s(k,m)| \frac{s^{k-m}}{k!}.
\end{equation*}
The change of variable $ n=k-m$ leads to
\begin{equation*}
F(m,s)=\left (\frac{\alpha }{A}\right )^{m} \sum ^{\infty }_{n=0}\frac{
\alpha ^{n} s^{n}}{n!}|s(m+n,m)| \frac{m!n!}{(m+n)!}.
\end{equation*}
In the Taylor theorem on the binomial expansion of $ F(m,s)$, there are
only finite numbers of terms:
\begin{equation*}
F(m,s)=\left (\frac{\alpha }{A}\right )^{m} \left [1+G(s)\right ]^{m}=
\left (\frac{\alpha }{A}\right )^{m} \sum ^{m}_{k=0} \frac{m!}{(m-k)!}
\frac{[G(s)]^{k}}{k!}.
\end{equation*}
After replacing $ \frac{[G(s)]^{k}}{k!}$ by the expansion (\ref{B}) we
change the summation order. Then the obtained result is
\begin{equation*}
F(m,s)=\left (\frac{\alpha }{A}\right )^{m} \sum ^{\infty }_{n=0}\frac{
\alpha ^{n} s^{n}}{n!} \sum ^{m\wedge n}_{k=0}\frac{m!}{(m-k)!} B_{n,k}(c
_{^{\bullet }}),
\end{equation*}
where $m\wedge n=\min \{m,n\} $. The p.m.f. $P(L(m)=n )$ is given by the
sequence of coefficients in front of $s^{n}$ in the p.g.f. $ F(m,s)$.
\end{proof}
The comparison of two expressions (\ref{Lm}) and (\ref{Lmn}) leads to
the following combinatorics identity:\index{combinatorics identity}
\begin{equation*}
\sum ^{m\wedge n}_{k=1}\frac{1}{(m-k)!} B_{n,k}(c_{\bullet })=|s(m+n,m)|
\frac{n!}{(m+n)!}.
\end{equation*}
%
%r1 #&#
\begin{remark}
{The harmonic numbers take part in the %development
expansion of the hypergeometric functions. A complete review on summation formulas involving generalized harmonic numbers and Stirling numbers\index{Stirling numbers} is given in
\cite{Sriv}. The generalized harmonic numbers are defined as follows:
\begin{equation*} H^{(k)}_{n}:=1+\frac{1}{2^{k}}+\frac{1}{3^{k}}+\cdots+\frac{1}{n^{k}},\quad H^{(1)}_{n}=H_{n}. \end{equation*}
For $ m=2,3,4,\ldots$, we use the following relations between Stirling numbers\index{Stirling numbers} of the first kind and generalized harmonic numbers to calculate directly the convolution probability of $L(m)$ and to confirm the previous combinatorics identity.\index{combinatorics identity} For example,
%e10 #&#
\begin{equation}\label{H}
|s(2+n,2)|=(n+1)!\left (1+\frac{1}{2}+\frac{1}{3}+\cdots+\frac{1}{n+1}\right )
\end{equation}
and
\begin{equation*} |s(n,3)|=\frac{(n-1)!}{2}\{(H_{n-1})^{2} -H^{(2)}_{n-1} \},\end{equation*}
and
\begin{equation*} |s(n,4)|=\frac{(n-1)!}{3!}\{(H_{n-1})^{3} -3 H_{n-1} H^{(2)}_{n-1}+2H^{(3)}_{n-1}\}.\end{equation*}
The general recurrence formula on this relation is given in \cite{Sriv}.
}
\end{remark}

%s3 #&#
\section{Transition probability\index{transition probability} and the Bernstein function\index{Bernstein function} of the Logarithmic L\'{e}vy process\index{L\'{e}vy process}}%
\label{sec3}
The principal information on the behaviour of any L\'{e}vy process\index{L\'{e}vy process} is
given by the representative random variable and it is expressed by the
canonical representation of the Bernstein function\index{Bernstein function} and the L\'{e}vy measure,\index{L\'{e}vy measure}
\cite{Sato,Sch}. The Laplace transform of the process $L(t)$ is given
by
\begin{equation*}
E[e^{-\lambda L(t)}]=\exp (-t\psi _{L}(\lambda )),\quad \lambda \geq 0,
\end{equation*}
where the Laplace exponent is a Bernstein function\index{Bernstein function} defined by the random
variable $ L(1)$ as follows:
\begin{equation*}
\psi _{L}(\lambda )=-\log \left (E[e^{-\lambda L(1)}]\right ),\quad
\lambda \geq 0.
\end{equation*}
For integer-valued L\'{e}vy processes\index{L\'{e}vy process} it is also convenient to work with
probability generating functions, see \cite{Ste}, Chapter 2.

Let us denote the L\'{e}vy measure\index{L\'{e}vy measure} of the process $L(t)$ by
$ \Pi _{L}(n), n=1, 2,\ldots $, its total mass by $\theta _{L}=\sum ^{\infty
}_{n=1}\Pi _{L}(n) $, and its generating function by
\begin{equation*}
Q_{L}(s)=\sum ^{\infty }_{n=1}s^{n} \Pi _{L}(n),\quad |s | \leq 1.
\end{equation*}
The normalised L\'{e}vy measure\index{L\'{e}vy measure} is denoted by $ \widetilde{\Pi }_{L}(n)$
and respectively, its p.g.f. as $\widetilde{Q}_{L}(s)= Q_{L}(s)/\theta
_{L}$. Then in these notations, the p.g.f. $F_{L}(t,s)=E[s^{L(t)}] $ is
given by
\begin{equation*}
F_{L}(t,s)=\exp \{-t \theta _{L} [1-\widetilde{Q}_{L}(s)] \}=\exp \{-t
\theta _{L} + t Q_{L}(s) \}, \quad |s| \leq 1,
\end{equation*}
and the Bernstein function\index{Bernstein function} is in the form
\begin{equation*}
\psi _{L}(\lambda )=\sum ^{\infty }_{k=1}\left (1-e^{-\lambda k}\right )
\Pi _{L}(k),\quad \lambda \geq 0.
\end{equation*}
All these characteristics of the Logarithmic L\'{e}vy process
$L(t)$\index{L\'{e}vy process} are specified in the following lemma.
%
%l2 #&#
\begin{lemma}
The L\'{e}vy measure\index{L\'{e}vy measure} of the process $L(t)$ generated by the infinitely
divisible logarithmic series\index{infinitely divisible logarithmic series} distribution $L(1)$ (\ref{LL}) is given for
$ n=1, 2,\ldots$ by the partial Bell polynomials\index{partial Bell polynomials} as follows:
%
%e11 #&#
\begin{equation}
\label{LM}
\Pi _{L}( n) = \frac{\alpha ^{n}}{n!}\sum ^{n}_{k=1} (-1)^{k-1}(k-1)! B
_{n,k}(c_{\bullet }),\quad c_{k}=\frac{ k!}{k+1}.
\end{equation}
The generating function of the L\'{e}vy measure\index{L\'{e}vy measure} is
\begin{equation*}
Q_{L}(s)=\log (1+G(s)), \quad \mbox{}_{2}F_{1}(1,1;2;\alpha s)=1+G(s),\quad |s|
\leq 1.
\end{equation*}
The Bernstein function\index{Bernstein function} of the Logarithmic L\'{e}vy process $L(t)$\index{L\'{e}vy process} is
\begin{equation*}
\psi _{L} (\lambda )=\theta _{L}\left \{  1 - \frac{1}{\theta _{L}}\log (1+G(e
^{-\lambda }) )\right \}  ,\quad \lambda \geq 0,
\end{equation*}
where
\begin{equation*}
\theta _{L} =\psi _{L} (\infty )= -\log \left (\frac{\alpha }{A}\right ).
\end{equation*}
\end{lemma}
\begin{proof}
Following representation (\ref{G}) of p.g.f. $ F(1,s)$, it is enough to
write
\begin{equation*}
\log (F(1,s))=\log \left (\frac{\alpha }{A} [1+G(s)]\right )= \log
\left (\frac{\alpha }{A}\right )+\log \left (1+G(s)\right )
\end{equation*}
in order to get the generating function of the L\'{e}vy measure.\index{L\'{e}vy measure} The total
mass of the L\'{e}vy measure\index{L\'{e}vy measure}
$\theta _{L}=-\log \left (\frac{\alpha }{A}\right ) $ because
$ G(0)=0$.

The logarithmic function $\log (1+x)$ is expanding by the signed
Stirling numbers\index{signed Stirling numbers} of the first kind and the expansion of $G(s)$ is given
previously in (\ref{B}). Then
%
%e12 #&#
\begin{equation}
\label{LG}
\log (1+G(s))=\sum ^{\infty }_{k=1}(-1)^{k-1}\frac{(G(s))^{k}}{k}=
\sum ^{\infty }_{k=1}(-1)^{k-1}(k-1)!\frac{(G(s))^{k}}{k!}.
\end{equation}
Exchanging the order of summation and in view of $B_{n,k}=0$, $k\geq n+1
$, we write
\begin{equation*}
\sum ^{\infty }_{k=1}(-1)^{k-1}(k-1)!\sum ^{\infty }_{n=k}B_{n,k}(g_{
\bullet })\frac{s^{n}}{n!}=\sum ^{\infty }_{n=1}\sum ^{n}_{k=1}(-1)^{k-1}(k-1)!B
_{n,k}(c_{\bullet }) \frac{\alpha ^{n} s^{n}}{n!}.
\end{equation*}
The L\'{e}vy measure\index{L\'{e}vy measure} is given by the sequence of coefficients in front
of $s^{n}$ in $ Q_{L}(s)$.
\end{proof}
As a direct result of (\ref{LM}), the computations of several terms of
the L\'{e}vy measure\index{L\'{e}vy measure} are simplified, such as
\begin{gather*}
\Pi _{L}(1)=\frac{\alpha }{2}, \quad \Pi _{L}(2)=\frac{\alpha ^{2}}{2!}
\frac{5}{12},\quad  \Pi _{L}(3)=\frac{\alpha ^{3}}{3!}\frac{3}{4},
\\
 \Pi _{L}(4)=\frac{
\alpha ^{4}}{4!}\frac{251}{120}, \quad \Pi _{L}(5)=\frac{\alpha ^{5}}{5!}
\frac{95}{12}.
\end{gather*}

It is well known from the \cite{Ste} (see Theorem 4.4, Chapter 2),
that the p.m.f. $P(L(1)=n)$, $n=0, 1,\ldots$, (\ref{LL}) is related to the
sequence of the (canonical) L\'{e}vy measure $ \Pi _{L}(n), n=1, 2,\ldots$,
(\ref{LM}) by the following recurrence equation:
\begin{equation*}
(n+1) P(L(1)=n+1)=\sum ^{n}_{k=0} P(L(1)=k)(n-k+1)\Pi _{L}(n-k+1).
\end{equation*}
It is equivalent to the next combinatorical identity:\index{combinatorical identity}
\begin{equation*}
\frac{n+1}{n+2}=\sum ^{n}_{k=0} \frac{1}{(k+1)(n-k)!}\sum ^{n-k+1}_{j=1}
(-1)^{j-1}(j-1)! B_{n-k+1,j}(c_{\bullet }).
\end{equation*}

There are two ways to define the transition probability\index{transition probability} $P(L(t)=n)$, $n=0,
1,\ldots$, of the L\'{e}vy process\index{L\'{e}vy process} $ L(t)$. We could proceed either by
starting with p.g.f. $ F_{L}(t,s)$ and its Taylor expansion or by using
the L\'{e}vy measure\index{L\'{e}vy measure} to define the compound Poisson process $L(t)$. We present
these methods separately in two independent proofs.
%
%t1 #&#
\begin{thm}\label{thm1}
Let $ \{L(t), t\geq 0 \}$ be a L\'{e}vy process\index{L\'{e}vy process} generated by the
infinitely divisible logarithmic series\index{infinitely divisible logarithmic series} distribution (\ref{LL})
supported by $ \{0, 1, 2,\ldots\}$. Then its transition probability\index{transition probability} is given for
$n=0, 1, 2,\ldots$ by
%
%e13 #&#
\begin{equation}
\label{P}
P(L(t)=n)= \left (\frac{\alpha }{A}\right )^{t} \frac{\alpha ^{n}}{n!}
\sum ^{n}_{k=0} [t]_{k \downarrow } B_{n,k}(c_{\bullet }),
\end{equation}
or equivalently:
%
%e14 #&#
\begin{equation}
\label{PP}
P(L(t)=n)= \left (\frac{\alpha }{A}\right )^{t} \frac{\alpha ^{n}}{n!}
\sum ^{n}_{k=0} t^{k} B_{n,k}(y_{\bullet }),\quad B_{0,0}=1,
\end{equation}
where
\begin{equation*}
y_{n}= \sum ^{n}_{k=1} (-1)^{k-1}(k-1)! B_{n,k}(c_{\bullet }) , \quad
c_{k}=\frac{k!}{k+1}.
\end{equation*}
\end{thm}
\begin{proof}[Proof 1]
 The transition probability $P(L(t)=n)$\index{transition probability} is the coefficient
in front of $ s^{n}$ in the expansion of p.g.f. $F_{L}(t,s)=\left (\frac{
\alpha }{A}\right )^{t} (1+G(s))^{t} $. The Taylor theorem for the
binomial expansion following (\ref{B}) leads to
\begin{equation*}
F_{L}(t,s)=\left (\frac{\alpha }{A}\right )^{t}
\sum ^{\infty }_{k=0}[t]_{k \downarrow } \sum ^{\infty }_{n=k} B_{n,k}(g
_{\bullet }) \frac{s^{n}}{n!}.
\end{equation*}
Then after exchanging the order of summation we find
\begin{equation*}
F_{L}(t,s)=\left (\frac{\alpha }{A}\right )^{t} \sum ^{\infty }_{n=0} \frac{
\alpha ^{n} s^{n}}{n!} \sum ^{n}_{k=0} [t]_{k \downarrow } B_{n,k}(c
_{\bullet }).
\end{equation*}
Because the partial Bell polynomials $ B_{0,0}=1$\index{partial Bell polynomials} and $B_{0,k}=0$, $k=1,2,\ldots $, the following result is valid:
\begin{equation*}
F_{L}(t,s)=\left (\frac{\alpha }{A}\right )^{t}\left ( 1+ \sum ^{
\infty }_{n=1} \frac{\alpha ^{n} s^{n}}{n!} \sum ^{n}_{k=1} [t]_{k
\downarrow } B_{n,k}(c_{\bullet })\right ),\quad c_{k}=\frac{k!}{k+1}.\qedhere
\end{equation*}
\end{proof}
In particular, it is easy to calculate several terms of the transition
probability,\index{transition probability} directly from (\ref{P}):
\begin{gather*}
P(L(t)=1)=\left (\frac{\alpha }{A}\right )^{t} \frac{\alpha t}{2},
\quad P(L(t)=2)=\left (\frac{\alpha }{A}\right )^{t}
\frac{\alpha ^{2}}{2!}\left (\frac{2t}{3}+\frac{t(t-1)}{4}\right ),
\\
P(L(t)=3)=\left (\frac{\alpha }{A}\right )^{t} \frac{\alpha ^{3}}{3!}
\left \{  \frac{3! t}{4}+\frac{t(t-1)3
.2!}{2.3}+\frac{t(t-1)(t-2)}{8}\right \}  ,
\\
P(L(t)=4)=\left (\frac{\alpha }{A}\right )^{t} \frac{\alpha ^{4}}{4!}
\left \{  \frac{4! t}{5}+\frac{13}{3}[t]_{2 \downarrow }+[t]_{3 \downarrow
}+ \frac{1}{16}[t]_{4 \downarrow }\right \}  .
\end{gather*}\newpage
\begin{proof}[Proof 2]
 Let the positive random variable $\xi $ be defined by the
normalised L\'{e}vy measure\index{L\'{e}vy measure} (\ref{LM}), having p.m.f.
\begin{equation*}
P(\xi =n)= \Pi _{L}(n)/\theta _{L},\quad n=1,2,\ldots,\quad \theta _{L}=
\log \Bigl(\frac{A}{\alpha }\Bigr),
\end{equation*}
and p.g.f. $E[s^{\xi }]= Q_{L}(s)/ \theta _{L}$. Let $( \xi _{1}, \xi
_{2}, \ldots , \xi _{k}, k=1,2,\ldots)$ be independent copies of the random
variable $\xi $. Following definition of the compound Poisson process,
the transition probability\index{transition probability} is represented as follows:
\begin{equation*}
P(L(t)=n)=\sum ^{\infty }_{k=0}e^{-\theta t}\frac{(\theta t)^{k}}{k!}P(
\xi _{1}+\xi _{2}+\cdots+ \xi _{k}=n ),\quad \theta =\log \Bigl(\frac{A}{\alpha
}\Bigr).
\end{equation*}
Taking into account (\ref{LG}) and (\ref{Bn}) we can represent the
function
\begin{equation*}
\frac{1}{\theta } Q_{L}(s)=\frac{1}{\theta }\log (1+G(s))
\end{equation*}
as an exponential generating function $\frac{1}{\theta } Q_{L}(s)=
\sum ^{\infty }_{n=1}\frac{x_{n} s^{n}}{n!}$,\index{exponential generating function} where
\begin{equation*}
x_{n}= \frac{1}{\theta }\alpha ^{n}\sum ^{n}_{k=1} (-1)^{k-1}(k-1)! B
_{n,k}(c_{\bullet }).
\end{equation*}
It means that the normalised probability convolution distribution\index{normalised probability convolution distribution}
\begin{equation*}
P(\xi _{1}+\xi _{2}+\cdots +\xi _{k}=n )= B_{n,k}(x_{\bullet })
\frac{k!}{n!}.
\end{equation*}
Then
\begin{equation*}
P(L(t)=n)=\sum ^{\infty }_{k=0}e^{-\theta t}\frac{(\theta t)^{k}}{k!} B
_{n,k}(x_{\bullet }) \frac{k!}{n!},\quad k\leq n.
\end{equation*}
The infinite sum is reduced to the finite one because $B_{n,k}(x_{
\bullet })=0 $ when $k>n $. We know that $ \theta =-\log (\frac{
\alpha }{A}) $ and $e^{-\theta t}=\left (\frac{\alpha }{A}\right )
^{t} $. Let us denote
\begin{equation*}
y_{n}=\sum ^{n}_{k=1} (-1)^{k-1}(k-1)! B_{n,k}(c_{\bullet }).
\end{equation*}
Then following the formula (\ref{Bn}), we obtain, for $ n=0,1,2,\ldots$,
\begin{equation*}
P(L(t)=n)=\left (\frac{\alpha }{A}\right )^{t}\sum ^{n}_{k=0} t^{k}\frac{
\alpha ^{n}}{n!} B_{n,k}(y_{\bullet }).
\end{equation*}
The probability $P(L(t)=0)=(\frac{\alpha }{A})^{t} $ corresponds to the
$ B_{0,0}(y_{\bullet })=1$.
\end{proof}
We remark that in the matrix representation of partial Bell polynomials\index{partial Bell polynomials}
for composition function the numbers $ B_{n,k}(x_{\bullet }), n\geq k
\geq 1$, are defined as product of matrices, \cite{Pit}, page 19.
Let us denote by $H(s)$ and $G(s)$ respectively the exponential
generating functions\index{exponential generating function} of both sequences $(h_{\bullet })$ and
$(g_{\bullet })$. Likewise, by $(x_{\bullet })$ is denoted the
sequence whose exponential generating function\index{exponential generating function} is the composition
$ H(G(s))$, such as
\begin{equation*}
H(s)=\log (1+s),\quad H(G(s))=\log (1+G(s)).
\end{equation*}
Then the matrix associated with the sequence $(x_{\bullet })$ is the
product of the triangular matrices associated with $(g_{\bullet })$ and
$(h_{\bullet })$ respectively:
\begin{equation*}
B_{n,k}(x_{\bullet })=\sum ^{n}_{j=k} B_{n,j}(g_{\bullet }) B_{j,k}(h
_{\bullet }),\quad k\leq j\leq n.
\end{equation*}
The sequence $(h_{\bullet })$ defined by the function $ H(s)=\log (1+s)$ is
exactly the sequence $ h_{k}=(-1)^{k-1}(k-1)!$ and $B_{n,k}(h_{\bullet
})=s(n,k)$, i.e. signed Stirling numbers\index{signed Stirling numbers} of the first kind. Then, after
applying formulas (\ref{Bn}) and (\ref{sn}) and changing the order of
summation, where $ k\leq j$, we confirm the equivalence of (\ref{P}) and
(\ref{PP}) as follows:
\begin{equation*}
\sum ^{n}_{k=1} t^{k} B_{n,k}(x_{\bullet })= \alpha ^{n} \sum ^{n}_{j=1}B
_{n,j}(c_{\bullet }) \sum ^{j}_{k=1} t^{k} B_{j,k}(h_{\bullet })=\alpha
^{n}\sum ^{n}_{j=1}B_{n,j}(c_{\bullet })[t]_{j \downarrow }.
\end{equation*}

The L\'{e}vy measure\index{L\'{e}vy measure} is the infinitesimal generator of the convolution
semi-group given by the transition probability\index{transition probability} $P(L(t)=n)$, $n=0, 1,
2,\ldots $, see \cite{Berg}, page 172. It is a limit in vague
convergence, see \cite{B}, page 39, as follows:
\begin{equation*}
\Pi _{L}(n) = \lim _{t \downarrow 0}\frac{P_{L}(t,n)}{t},\quad n=1,2,\ldots
.
\end{equation*}
Then
\begin{equation*}
\Pi _{L}(n)=\lim _{t \downarrow 0}\Bigl(\frac{\alpha }{A}\Bigr)^{t} \frac{\alpha
^{n}}{n!}\sum ^{n}_{k=1} \frac{[t]_{k \downarrow }}{t} B_{n,k}(c_{
\bullet }).
\end{equation*}
Finally, we know that $ [t]_{k \downarrow }= [-t]_{ k
\uparrow }(-1)^{k}$. In this way,
\begin{equation*}
\lim _{t \downarrow 0} \frac{[t]_{k \downarrow }}{t}=(-1)^{k-1}(k-1)!.
\end{equation*}

%s4 #&#
\section{Negative-Binomial process subordinated by the Gamma process}%
\label{sec4}
The concept of subordination\index{subordination} was introduced by S. Bochner in 1955 for
the Markov processes, L\'{e}vy processes,\index{L\'{e}vy process} and corresponding semigroups,
as randomization of the time parameter: $Y(t)=X(T(t))$. There are two
sources of randomness -- the underlying process $X(t)$ and a random time
process $T(t)$, under the assumption of their independence. The
time-change process $T(t)$ is supposed to be a subordinator -- the L\'{e}vy
process\index{L\'{e}vy process} with nonnegative increments, \cite{B}, Chapter 3. The
independence of the ground process\index{ground process} and the random time process ensures
the preservation of Markov property and L\'{e}vy property for the
subordinated process.\index{subordinated process} The transformation of the main probabilistic
characteristics, such as transition probability,\index{transition probability} L\'{e}vy measure\index{L\'{e}vy measure} and
Laplace exponent, is stated and proved in \cite{Sato}, Chapter 6,
Theorem 30.1. See also \cite{Ci}, Chapter 7, Theorem 6.2 and
Theorem 6.18. They are our principal references.

In this paragraph, we study the effect of a random time-change for the
Negative-Binomial process $\{X(t), t\geq 0\}$. The L\'{e}vy measure\index{L\'{e}vy measure} of
a Negative-Binomial process is defined by a logarithmic series\index{logarithmic series} distribution
supported by positive integers $ N=\{1, 2, \ldots\}$ with the same
parameter $0< \alpha < 1$ as for the Logarithmic L\'{e}vy process\index{L\'{e}vy process}
$L(t)$. The Gamma subordinator\index{Gamma subordinator} $ \{T_{\beta }(t), t\geq 0 \}$ with
selective parameter $\beta >0$ can be considered as a random observation
time, where the mathematical expectation $ E[T_{\beta }(t)]=\beta t$.
The obtained results are formulated and proved in the following theorem.
%
%t2 #&#
\begin{thm}\label{thm2}
Let $ \{X(t), t\geq 0 \}$ be a Negative-Binomial process with the Bernstein
function\index{Bernstein function}
\begin{equation*}
\psi _{X}(\lambda )=\log \left ( \frac{1-\alpha e^{-\lambda }}{1-
\alpha } \right ),\quad \lambda \geq 0,\quad \psi _{X}(\infty )=-
\log (1-\alpha )=A.
\end{equation*}
Let $ \{T_{\beta }(t), t\geq 0 \}$ be a Gamma subordinator\index{Gamma subordinator} with
the Bernstein function\index{Bernstein function}
\begin{equation*}
\psi _{T}(\lambda )=\log (1+\beta \lambda ),\quad \lambda \geq 0,
\quad \psi _{T}(\infty )=\infty .
\end{equation*}
Suppose the processes $ X(t)$ and $ T_{\beta }(t)$ are independent.

Then for the subordinated process $ Y(t)=X(T_{\beta }(t))$ the
following results are valid.

The Bernstein function\index{Bernstein function} of $Y(t)$ is given by
\begin{equation*}
\psi _{Y}(\lambda )=\log \left (1+\beta \log \left ( \frac{1-\alpha e
^{-\lambda }}{1-\alpha }\right )\right ),\quad \lambda \geq 0,\quad
\psi _{Y}(\infty )=\log (1+A\beta ).
\end{equation*}
The L\'{e}vy measure\index{L\'{e}vy measure} of the subordinated process\index{subordinated process} is given by
\begin{equation*}
\Pi _{Y}(n)= \frac{\alpha ^{n}}{n!}\sum ^{n}_{k=1}|s(n,k)| (k-1)!\left (\frac{
\beta }{1+A\beta }\right )^{k} ,\quad n=1,2,\ldots.
\end{equation*}
The transition probability\index{transition probability} $P(Y(t)=n)$, $n=0,1,2,\ldots$, is given by
\begin{equation*}
P(Y(t)=n)=
\frac{\alpha ^{n}}{n!}\frac{1}{(1+A\beta )^{t}}\sum ^{n}_{k=0}|s(n,k)| [t]_{k
\uparrow } \left (\frac{\beta }{1+A\beta }\right )^{k},
\end{equation*}
or equivalently:
\begin{equation*}
P(Y(t)=n)= \frac{\alpha ^{n}}{n!}\frac{1}{(1+A\beta )^{t}} \sum ^{n}
_{k=0} t^{k} B_{n,k}(w _{\bullet }) ,
\end{equation*}
where the sequence $ (w _{\bullet })$ is defined by
\begin{equation*}
w_{n}= \sum ^{n}_{k=1}|s(n,k)| |s(k,1)|\left (\frac{ \beta }{1+A\beta
}\right )^{k}, \quad |s(k,1)|=(k-1)!.
\end{equation*}
\end{thm}
\begin{proof}
The main assumption in the definition of subordination\index{subordination} by Bochner is the
independence of the ground process\index{ground process} and the random time-change process.
The methods of the Laplace transform and conditional probability for
independent processes give the following convenient representations of
the main characteristics, see \cite{Sato}, page 197. The Bernstein
function\index{Bernstein function} of the subordinated process\index{subordinated process} is the composition of the
corresponding Bernstein functions,\index{Bernstein function} as follows:
\begin{equation*}
\psi _{Y}(\lambda )=\psi _{T_{\beta }}(\psi _{X}(\lambda )).
\end{equation*}
The transition probability\index{transition probability} of the subordinated process\index{subordinated process} is expressed by
the conditional probability and is given as the integral of transition
probability\index{transition probability} of the ground process\index{ground process} by the transition probability of the Gamma
subordinator:\index{Gamma subordinator} \index{transition probability}
\begin{align*}
P(Y(t)=n)&=\int ^{\infty }_{0+} P(X(u)=n)u^{t-1} e^{- u/\beta } \frac{du}{
\beta ^{t} \Gamma (t) }
\\
&= \int ^{\infty }_{0+} (1-\alpha )^{u}[u]_{n \uparrow } \frac{\alpha
^{n}}{n!}u^{t-1} e^{- u/\beta } \frac{du}{\beta ^{t} \Gamma (t)}.
\end{align*}
Replacing the increasing factorials (\ref{sn}) $[u]_{n \uparrow }=
\sum ^{n}_{k=0} |s(n,k)|u^{k}$ and
\begin{equation*}
(1-\alpha )^{u}=e^{u\log (1-\alpha )}=e^{-Au},
\end{equation*}
we obtain
\begin{align*}
P(Y(t)=n)&= \frac{\alpha ^{n}}{n!}\sum ^{n}_{k=0} |s(n,k)| \int ^{\infty
}_{0+} e^{-Au} e^{- u/\beta } u^{k+t-1}
\frac{du}{\beta ^{t} \Gamma (t)}
\\
&= \frac{\alpha ^{n}}{n!}\sum ^{n}_{k=0} |s(n,k)| \frac{\Gamma (t+k)\beta
^{k}}{\Gamma (t)(1+A\beta )^{t+k}}
\\
&=\frac{\alpha ^{n}}{n!}\frac{1}{(1+A\beta )^{t}}\sum ^{n}_{k=0}|s(n,k)|
[t]_{k\uparrow } \left (\frac{\beta }{1+A\beta }\right )^{k}.
\end{align*}
Let us remark, that the L\'{e}vy measure\index{Gamma subordinator} \index{L\'{e}vy measure} of the Gamma subordinator in our
parametrisation is given by
\begin{equation*}
\Pi _{T_{\beta }}(du)=e^{- u/\beta } du/u ,
\end{equation*}
see \cite{B}, page 73. Then, from the results proved in
\cite{Sato,Ci}, the L\'{e}vy measure\index{L\'{e}vy measure} of the subordinated process\index{subordinated process} can be
calculated as the integral of transition probability\index{transition probability} of the ground
process\index{ground process} by the L\'{e}vy measure of the Gamma subordinator:\index{Gamma subordinator} \index{L\'{e}vy measure}
\begin{align*}
\Pi _{Y}(n)&=\int ^{\infty }_{0+} P(X(u)=n) e^{- u/\beta } \frac{du}{u}=
\int ^{\infty }_{0+} (1-\alpha )^{u}[u]_{n \uparrow }
\frac{\alpha ^{n}}{n!} e^{- u/\beta } \frac{du}{u}
\\
&= \frac{\alpha ^{n}}{n!}\sum ^{n}_{k=1} |s(n,k)| \int ^{\infty }_{0+} e
^{-Au} e^{- u/\beta } u^{k} \frac{du}{u}
\\
&= \frac{\alpha ^{n}}{n!}\sum
^{n}_{k=1} |s(n,k)|\Gamma (k) \left (\frac{\beta }{1+A\beta }\right )
^{k}.
\end{align*}
From the Bernstein function $\psi _{Y}(\lambda ) $\index{Bernstein function} we derive the
generating function of the L\'{e}vy measure $\Pi _{Y}(n)$ in the
following form:
\begin{equation*}
Q_{Y}(s)=\theta _{Y} \widetilde{Q}_{Y}(s)=- \log \left (1-\frac{\beta
}{1+A\beta }\{-\log (1-\alpha s)\}\right ).
\end{equation*}
It can be presented as an exponential generating function\index{exponential generating function} as follows:
\begin{equation*}
Q_{Y}(s)=\sum ^{\infty }_{n=1} u_{n} \frac{s^{n}}{n!}, \quad u_{n}=
\sum ^{n}_{k=1}B_{n,k}(v_{\bullet })B_{k,1}(s_{\bullet }),
\end{equation*}
where the sequences $(v_{\bullet })$ and $(s_{\bullet })$ are defined
respectively by
\begin{equation*}
v_{k}=\frac{k! \alpha ^{k}}{k},\quad s_{k}=\frac{k!}{k}\left (\frac{
\beta }{1+a\beta }\right )^{k} .
\end{equation*}
Moreover,
\begin{equation*}
B_{n,k}(v_{\bullet })=\alpha ^{n} |s(n,k)|,\quad B_{k,1}(s_{\bullet })=
\left (\frac{\beta }{1+A\beta }\right )^{k} |s(k,1)|
\end{equation*}
and
\begin{equation*}
u_{n}=\alpha ^{n} \sum ^{n}_{k=1}|s(n,k)| |s(k,1)|\left (\frac{\beta }{1+A
\beta }\right )^{k}.
\end{equation*}
Let
\begin{equation*}
( \xi _{1}, \xi _{2}, \ldots , \xi _{k}, k=1,2,\ldots)
\end{equation*}
be independent copies of the positive random variable $\xi $ with p.m.f.
\begin{equation*}
P(\xi =n)=\Pi _{Y}(n)/\theta ,\quad n=1,2,\ldots,\quad \theta =\theta _{Y}=
\log (1+A\beta ).
\end{equation*}
In a complete analogy with the Proof 2 of Theorem~\ref{thm1}, we find the
normalised probability convolution distribution\index{normalised probability convolution distribution}
\begin{equation*}
P(\xi _{1}+\xi _{2}+\cdots+ \xi _{k}=n )= \frac{1}{\theta ^{k}}B_{n,k}(u_{
\bullet })\frac{k!}{n!}.
\end{equation*}
Then for $ n=0,1,2,\ldots$, we have:
\begin{equation*}
P(Y(t)=n)=\sum ^{\infty }_{k=0}e^{-\theta t}\frac{(\theta t)^{k}}{k!} \frac{1}{
\theta ^{k}} B_{n,k}(u_{\bullet })\frac{k!}{n!},\quad B_{0,0}=1.
\end{equation*}
Obviously, the exponential decay is $e^{-\theta t}=(\frac{1}{1+A
\beta })^{t} $. Additionally, from the formula (\ref{Bn}) we derived
that
\begin{equation*}
B_{n,k}(u_{\bullet })= \alpha ^{n} B_{n,k}(w _{\bullet }),\quad w_{n}=
\sum ^{n}_{k=1}|s(n,k)| |s(k,1)|\left (\frac{\beta }{1+a\beta }\right )
^{k} .
\end{equation*}
So, taking into account, that $B_{n,k}(u _{\bullet })=0 $ for all
$ k>n$, we see that the infinite sum is reduced to the finite one
\begin{equation*}
P(Y(t)=n)= \frac{\alpha ^{n}}{n!}\frac{1}{(1+A\beta )^{t}} \sum ^{n}
_{k=0} t^{k} B_{n,k}(w _{\bullet }),\quad B_{0,0}=1.\qedhere
\end{equation*}
\end{proof}
Finally, we remark that for $n=1,2,\ldots$,
\begin{equation*}
\lim _{t \downarrow 0}\frac{ P(Y(t)=n)}{t}=\frac{\alpha ^{n}}{n!}B_{n,1}(w
_{\bullet })= \frac{\alpha ^{n}}{n!} w_{n}=\Pi _{Y}(n).
\end{equation*}

%s5 #&#
\section{Logarithmic L\'{e}vy process subordinated by the Poisson process}%
\label{sec5}
The next studied process $\{Z(t), t\geq 0\}$ is constructed as a random
time-change of the Logarithmic L\'{e}vy process $L(t)$\index{L\'{e}vy process} with the Poisson one in
the assumption of their independence. The selective parameter $b>0$ of
the Poisson process $ \{T_{b}(t), t\geq 0 \}$ is introduced, such as
mathematical expectation $E[T_{b}(t)]=bt $. The results are formulated and
proved in the following theorem.
%
%t3 #&#
\begin{thm}\label{thm3}
Let $ \{L(t), t\geq 0 \}$ be a Logarithmic L\'{e}vy process\index{L\'{e}vy process} with
the Bernstein function\index{Bernstein function}
\begin{equation*}
\psi _{L}(\lambda )=\log \left (\frac{ Ae^{-\lambda }}{-\log (1-\alpha
e^{-\lambda })}\right ),\quad \lambda \geq 0,\quad \psi _{L}(\infty )=
\log \left (\frac{A}{\alpha }\right )>0.
\end{equation*}
Let $ \{T_{b}(t), t\geq 0\}$ be a Poisson subordinator with the Bernstein
function\index{Bernstein function}
\begin{equation*}
\psi _{T}(\lambda )=b(1-e^{-\lambda }),\quad \lambda \geq 0,\quad \psi
_{T}(\infty )=b>0.
\end{equation*}
Suppose the processes $ L(t)$ and $ T_{b} (t)$ are independent.

Then for the subordinated L\'{e}vy process $ Z(t)=L(T_{b}(t))$\index{subordinated L\'{e}vy process}
the following results are valid.

The Bernstein function\index{Bernstein function} of the subordinated process $ Z(t)$ is given by
\begin{equation*}
\psi _{Z}(\lambda )=b \left (1+\frac{\log (1-\alpha e^{-\lambda })}{A e
^{-\lambda }}\right ),\quad \lambda \geq 0,\quad \psi _{Z}(\infty )=b
\left (1-\frac{\alpha }{A}\right )>0.
\end{equation*}
The L\'{e}vy measure\index{L\'{e}vy measure} of $ Z(t)$ is given by
\begin{equation*}
\Pi _{Z}( n) = \frac{b\alpha ^{n+1}}{A(n+1)},\quad n=1,2,\ldots.
\end{equation*}
The transition probability\index{transition probability} of the subordinated process $Z(t)$ is, for
$ n=0,1,\ldots$,
%
%e15 #&#
\begin{equation}
\label{Z}
P(Z(t)=n)=e^{-\theta t}\frac{\alpha ^{n} }{n!}
\sum ^{n}_{k=0}\left (\frac{ \alpha bt}{A}\right )^{k} B_{n,k}(c_{
\bullet }),\quad  c_{k}=\frac{k!}{k+1}, \quad \theta =b-\frac{\alpha b}{A}.
\end{equation}
\end{thm}
\begin{proof}
Once again, the composition of two Bernstein functions\index{Bernstein function} is obvious:
\begin{equation*}
\psi _{Z}(\lambda )=b (1-\exp (-\psi _{L}(\lambda \xch{))).}{)).}
\end{equation*}
The L\'{e}vy measure\index{L\'{e}vy measure} of the subordinated process\index{subordinated process} is given by the
following infinite sum, as it is shown in \cite{Sato,Ci},
\begin{equation*}
\Pi _{Z}( n) =\sum ^{\infty }_{k=1} P(L(k)=n) \Pi _{T}(k)= b P(L(1)=n)=
\frac{b}{A}\frac{\alpha ^{n+1}}{(n+1)},\quad n=1,2,\ldots,
\end{equation*}
because the normalised L\'{e}vy measure $ \widetilde{\Pi }_{L}(n),
n=1,2,\ldots $, of the Poisson process is exactly the delta function in
$ n=1 $. The total mass of the L\'{e}vy measure\index{L\'{e}vy measure} $ \Pi _{Z}(n),
n=1,2,\ldots$, is calculated directly from (\ref{LL}) as $ \theta _{Z}=
\frac{b}{A}(A-\alpha )$. The exponential generating function\index{exponential generating function} of the
L\'{e}vy measure $ \Pi _{Z}$ is given by (\ref{G}) and (\ref{GE}) as
follows:
\begin{equation*}
Q_{Z}(s)=\frac{b \alpha }{A}G(s)=\frac{b \alpha }{A}\sum ^{\infty }
_{k=1}g_{k} \frac{s^{k}}{k!},\quad g_{k}=\frac{\alpha ^{k} k!}{k+1}.
\end{equation*}
Let
\begin{equation*}
( \xi _{1}, \xi _{2}, \ldots , \xi _{k}, k=1,2,\ldots)
\end{equation*}
be independent copies of the positive random variable $\xi $ with p.m.f.
\begin{equation*}
P(\xi =n)=\Pi _{Z}(n) /\theta ,\quad n=1,2,\ldots,\quad \theta =\theta
_{Z}=b-\frac{b \alpha }{A}.
\end{equation*}
The normalised probability convolution distribution\index{normalised probability convolution distribution} is given by
\begin{equation*}
P(\xi _{1}+\xi _{2}+\cdots + \xi _{k}=n )=\left (\frac{\alpha }{A-\alpha }\right )
^{k} B_{n,k}(c_{\bullet }) \frac{\alpha ^{n} k!}{n!},\quad \frac{
\alpha }{A-\alpha }=\frac{b\alpha }{A\theta } .
\end{equation*}
Then the elementary transformations lead to (\ref{Z}) as follows:
\begin{align*}
P(Z(t)=n)&=\sum ^{\infty }_{k=0}e^{-\theta t}\frac{(\theta t)^{k}}{k!}
\left (\frac{\alpha }{A-\alpha }\right )^{k} B_{n,k}(c_{\bullet }) \frac{
\alpha ^{n} k!}{n!}
\\
&=e^{-\theta t}\frac{\alpha ^{n} }{n!} \sum ^{n}_{k=0} \left (\frac{
\alpha \theta t}{A-\alpha }\right )^{k} B_{n,k}(c_{\bullet })=e^{-
\theta t}\frac{\alpha ^{n} }{n!} \sum ^{n}_{k=0} \left (\frac{\alpha bt}{A}\right )
^{k} B_{n,k}(c_{\bullet }).
\end{align*}
In particular,
\begin{equation*}
P(Z(t)=0)=e^{-\theta t},\quad P(Z(t)=1)= \frac{\alpha ^{2} b t e^{-
\theta t}}{2A}.
\end{equation*}
Knowing that $ B_{2,1}=c_{2}=\frac{2!}{3}$ and $ B_{2,2}=(c_{1})^{2}=
\frac{1}{4}$ we find
\begin{equation*}
P(Z(t)=2)=e^{-\theta t}\frac{\alpha ^{2}}{2!}\left \{
\frac{2 \alpha bt}{3A}+ \left (\frac{\alpha bt}{2A}\right )^{2}
\right \}  .
\end{equation*}
In the same way, as $ B_{3,1}=\frac{3!}{4}$, $ B_{3,2}=1 $ and
$ B_{3,3}=\frac{1}{8}$ we obtain
%
%e16 #&#
\begin{equation}
\label{Z3}
P(Z(t)=3)=e^{-\theta t}\frac{\alpha ^{3}}{3!}\left \{
\frac{ \alpha bt}{A}\frac{3!}{4}+ \left (\frac{ \alpha bt}{A}\right )
^{2} + \left (\frac{ \alpha bt}{2A}\right )^{3}\right \}  ,
\end{equation}
and so on.
\end{proof}

%r2 #&#
\begin{remark}
{In this situation, the range of the random time process $ T_{b}(t)$
is a discrete integer-valued set $ Z_{+}=\{0, 1, 2,\ldots \}$. The
subordination\index{subordination} by Bochner gives the transition probability\index{transition probability} of the subordinated
process $ Z(t)=L(T_{b}(t))$ as the following conditional probability:
%e17 #&#
\begin{equation}\label{St}
P(Z(t)=n)= \sum ^{\infty }_{k=0}P(L(k)=n)e^{-bt} \frac{(bt)^{k}}{k!}.
\end{equation}
The transition probability\index{transition probability} of the ground  process $L(t)$\index{ground process} for integer-valued
time $t=k$ is given by the $ k$-fold convolution of the representative
random variable $ L(1)$, as the two equivalent expressions (\ref{Lmn}) and
(\ref{Lm}).}

%When
After replacing $P(L(k)=n)$ in (\ref{St}) by (\ref{Lmn}),  it is enough to
exchange the order of summation to prove (\ref{Z}), as follows:
\begin{align*}
P(Z(t)=n)&=  \sum ^{\infty }_{k=0}\left (\frac{\alpha }{A}\right )^{k} \frac{\alpha ^{n} }{n!} \sum ^{k\wedge n}_{j=0}\frac{k!}{(k-j)!} B_{n,j}(c_{\bullet })\frac{(bt)^{k}
e^{-bt}}{k!}
\\
&= \frac{\alpha ^{n} e^{-bt}}{n!} \sum ^{n}_{j=0}\left
(\frac{\alpha b t}{A}\right )^{j} B_{n,j}(c_{\bullet })\sum ^{\infty
}_{k=j}\left (\frac{\alpha b t}{A}\right )^{k-j} \frac{1}{(k-j)!}.
\end{align*}

But, if we take $P(L(k)=n)$ from (\ref{Lm}), then replacing it in (\ref{St})
we obtain
%e18 #&#
\begin{equation}\label{Zb}
P(Z(t)=n)= e^{-bt}\frac{\alpha ^{n}}{n!} \sum ^{\infty }_{k=0} \left (\frac{\alpha b t}{A}\right )^{k} |s(k+n,k)|\frac{n!}{(n+k)!}.
\end{equation}
The relation of the Stirling numbers\index{Stirling numbers} on the binomial coefficients explains the equivalence of (\ref{Zb}) to (\ref{Z}) and the presence of $e^{-bt} \neq e^{-\theta t}$ in (\ref{Zb}). For $ n=1 $ we have
\begin{equation*}|s(k+1,k)|=\frac{(k+1)!}{(k-1)! 2!},\quad |s(1,0)|=0.\end{equation*}
Then
\begin{align*}
P(Z(t)=1)&=\alpha e^{-bt} \sum ^{\infty }_{k=1}
\left (\frac{\alpha b t}{A}\right )^{k} |s(k+1,k)|\frac{1!}{(1+k)!}
\\
&=\alpha e^{-bt}\frac{\alpha b t}{ 2 A} \sum ^{\infty }_{k=1}\frac{1}{(k-1)!}
\left (\frac{\alpha b t}{A}\right )^{k-1} = e^{-bt}e^{\frac{\alpha b t}{A}} \frac{\alpha ^{2} b t}{ 2 A}.
\end{align*}
For $ n=2 $ we have
\begin{equation*}
|s(k+2,k)|=\frac{[3(k+2)-1]}{4}\frac{(k+2)!}{(k-1)! 3!},\quad |s(2,0)|=0,\quad |s(3,1)|=2.
\end{equation*}
We calculate
\begin{equation*}
P(Z(t)=2)=\frac{\alpha ^{2}}{2!} e^{-bt} \sum ^{\infty }_{k=1}\left (\frac{\alpha b t}{A}\right )^{k}\left ( \frac{3k+5}{4}\right )\left (\frac{2!}{(k -1)! 3!}\right ).\end{equation*}
Obviously,
\begin{equation*}\left ( \frac{3k+5}{4}\right )\left (\frac{2!}{(k -1)! 3!}\right )=\frac{1}{4(k-2)!}+\frac{2}{3(k-1)!},\quad k=2,3,\ldots. \end{equation*}
It means that the probability $P(Z(t)=2)$ is equal to:
\begin{equation*}
\frac{\alpha ^{2} e^{-bt}}{2!}\left \{  \frac{1}{4}\left (\frac{\alpha b t}{A}\right )^{2} \sum ^{\infty }_{k=2}\left (\frac{\alpha b t}{A}\right )^{k-2}\frac{1}{(k -2)!} +\frac{2}{3}\frac{\alpha b t}{A}\sum ^{\infty }_{k=1}\left (\frac{\alpha b t}{A}\right )^{k-1}\frac{1}{(k -1)!}\right \}.
 \end{equation*}
Finally,
\begin{equation*}P(Z(t)=2)=\frac{\alpha ^{2}}{2!} e^{-bt}e^{\frac{\alpha b t}{A}}\left \{  \frac{1}{4}\left (\frac{\alpha b t}{A}\right )^{2} +\frac{2}{3}\frac{\alpha b t}{A}\right \}  . \end{equation*}
For $ n=3$ we use the consecutive relations of the Stirling numbers\index{Stirling numbers} on the binomial coefficients:
\begin{equation*}|s(k+3,k)|=\frac{(k+3)!}{(k+1)! 2!} \frac{(k+3)!}{(k-1)! 4!},\quad |s(3,0)|=0,\quad |s(3,2)|=3 .\end{equation*}
The equivalent representation of $P(Z(t)=3)$
will be transformed as follows:
\begin{equation*}P(Z(t)=3)=\frac{\alpha ^{3}}{3!}e^{-bt}\sum ^{\infty }_{k=1}\left (\frac{\alpha b t}{A}\right )^{k} \frac{(k+3)(k+2)}{8(k-1)!}.\end{equation*}
Obviously,
\begin{equation*}\frac{(k+3)(k+2)}{8(k-1)!}=\frac{1}{8(k-3)!} +\frac{1}{(k-2)!} +\frac{3!}{4 (k-1)!},\quad k=3,4,\ldots. \end{equation*}
In the same way we obtain (\ref{Z3}),
and so on. The Stirling numbers\index{Stirling numbers} are very convenient in applications because they have recurrence relation and tables of their values.
\end{remark}
We confirm the expression of the L\'{e}vy measure\index{L\'{e}vy measure} by the following
limit:
\begin{equation*}
\lim _{t \downarrow 0}\frac{ P(Z(t)=n)}{t}=\frac{\alpha ^{n}}{n!} \frac{
\alpha b}{A}B_{n,1}(c _{\bullet })= \frac{b \alpha ^{n+1}}{An!}
\frac{n!}{n+1}=\Pi _{Z}(n),\quad n=1,2,\ldots.
\end{equation*}

%s6 #&#
\section{Applications}%
\label{sec6}
An important problem in many applications is how to recognize the
original process from the observation data when the registration is
randomly perturbed. The problem is growing in cases when the process is
composed of several different processes. We see that the probabilistic
characteristics for the couples of processes $ Y(t)$ and $L(t)$ as well
as for $Z(t)$ and $ X(t)$ are similar. The best way to demonstrate their
different properties is the comparison between Bernstein functions\index{Bernstein function} and
L\'{e}vy measures\index{L\'{e}vy measure} with different parameters. The Bernstein functions
$ \psi _{L}(\lambda )$\index{Bernstein function} and $ \psi _{Y}(\lambda )$, how they are defined
in Theorems~\ref{thm2} and \ref{thm3}, contain the iterated logarithmic function and have
the following derivatives at zero:
\begin{equation*}
\psi '_{L}(\lambda )=\frac{\alpha e^{-\lambda }}{(1-\alpha e^{-\lambda
})(-\log (1-\alpha e^{-\lambda }\xch{))}{)} }-1,\quad \psi '_{L}(0)=\frac{
\alpha }{A(1-\alpha )}-1
\end{equation*}
and
\begin{gather*}
\psi '_{Y}(\lambda )=\frac{\alpha \beta e^{-\lambda }}{\{1+A \beta +
\beta \log (1-\alpha e^{-\lambda })\}(1-\alpha e^{-\lambda })},
\\
\psi '_{Y}(0)=\frac{\alpha \beta }{ (1+A\beta -A\beta )(1-\alpha )}=\frac{
\alpha \beta }{1-\alpha }.
\end{gather*}
The first cumulants are equal to the corresponding mathematical
expectations $ E[L(1)]$ or $ E[Y(1)]$ and for the subordinated processes\index{subordinated process}
we have
\begin{equation*}
E[Y(1)]= E[X(1)] E[T_{\beta }(1)].
\end{equation*}
Similarly, the Bernstein functions\index{Bernstein function} for processes $X(t)$ and
$Z(t)$ are denoted by $ \psi _{X}(\lambda )$ and $ \psi _{Z}(\lambda )$
as in Theorems~\ref{thm2} and \ref{thm3}. Their derivatives at zero are as follows:
\begin{equation*}
\psi '_{X}(\lambda )=\frac{\alpha e^{-\lambda }}{1-\alpha e^{-\lambda
}},\quad \psi '_{X}(0)=\frac{\alpha }{1-\alpha }
\end{equation*}
and
\begin{gather*}
\psi '_{Z}(\lambda )=\frac{b}{A} \frac{\alpha e^{-\lambda }+(1-\alpha
e^{-\lambda })\log (1-\alpha e^{-\lambda })}{ e^{-\lambda }(1-\alpha
e^{-\lambda })},
\\
\psi '_{Z}(0)=\frac{b}{A} \frac{\alpha +(1-\alpha )\log (1-\alpha )}{1-
\alpha }=\frac{b}{A}\left (\frac{\alpha }{1-\alpha } -A\right ).
\end{gather*}
For the application, we constructed two different case tests, named as
\textit{Selection A} and \textit{Selection B}.

\textit{Selection A.} We can choose the parameters $\beta $ and $b$ in
a way that the corresponding total masses of the L\'{e}vy measures\index{L\'{e}vy measure} are
equal, $\theta _{L}=\theta _{Y} $, and $\theta _{X}=\theta _{Z} $, as
follows:
\begin{equation*}
\beta =\frac{A-\alpha }{A\alpha } <1,\quad b=\frac{A^{2}}{A-\alpha }>1
.
\end{equation*}
By this choice (selection) of parameters $ \beta $ and $b$ the
mathematical expectations are in the following inequalities:
\begin{equation*}
\psi '_{L}(0)< \psi '_{Y}(0)< \psi '_{X}(0)< \psi '_{Z}(0).
\end{equation*}
Namely,
\begin{equation*}
\frac{\alpha }{A(1-\alpha )}-1< \frac{ (A-\alpha )}{A(1-\alpha )}<\frac{
\alpha }{1-\alpha }< \frac{A}{A-\alpha }\left (\frac{\alpha }{1-
\alpha } -A\right ).
\end{equation*}
The values of the L\'{e}vy measures\index{L\'{e}vy measure} $\Pi _{X}$ and $\Pi _{Z}$ at
$ n=1$ satisfy the following inequality $ \Pi _{Z}(1)<\Pi _{X}(1)$, when
$ \sum ^{\infty }_{n=1}\Pi _{Z}(n) = \sum ^{\infty }_{n=1}\Pi _{X}(n)=A$,
as it is demonstrated in Figure \ref{f1} and in Figure \ref{f2}.
%
%f1 #&#
\begin{figure}%[b!]
\includegraphics{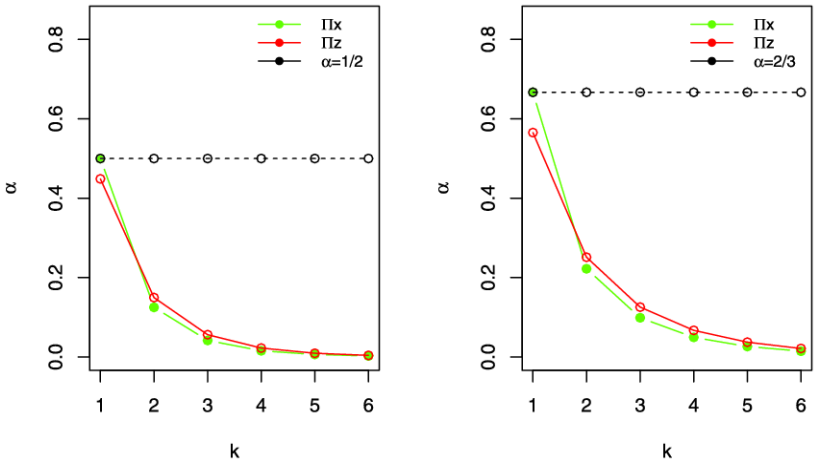}
\caption{Selection A. Comparison between the L\'{e}vy measure $\Pi _{X}$ and $\Pi _{Z}$ for $\alpha =1/2$ (left) and $\alpha =2/3$ (right), where $ \sum ^{\infty }_{n=1}\Pi _{Z}(n) = \sum ^{\infty }_{n=1}\Pi _{X}(n)=A $}
\label{f1}
\end{figure}
%f2 #&#
\begin{figure}
\includegraphics{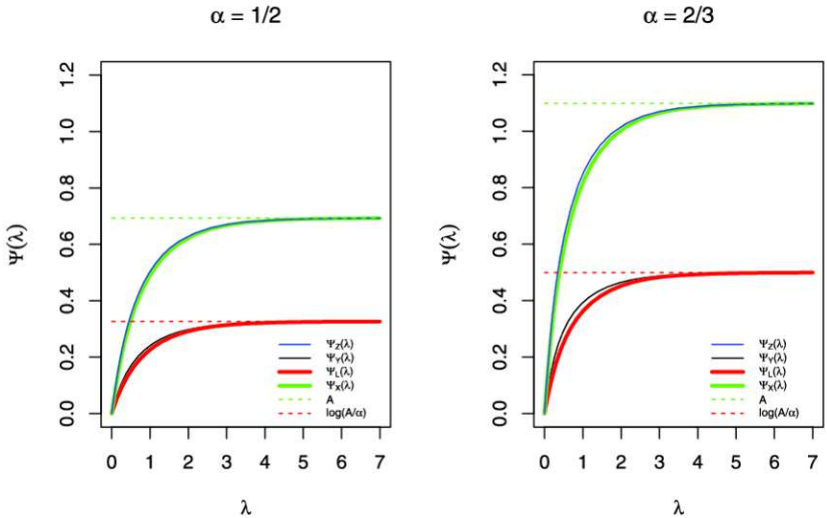}
\caption{Selection A. Comparative plot of main Bernstein functions after rescaling with $\alpha $ equal to $1/2$ and $2/3$, where $\psi '_{L}(0) < \psi '_{Y}(0)< \psi '_{X}(0)< \psi '_{Z}(0)$, knowing that $\theta _{L}=\theta _{Y}=\log (\frac{A}{\alpha })$ and $\theta _{X}=\theta _{Z}=A=-\log (1-\alpha ) $}
\label{f2}
\end{figure}
%
%f3 #&#
\begin{figure}
\includegraphics{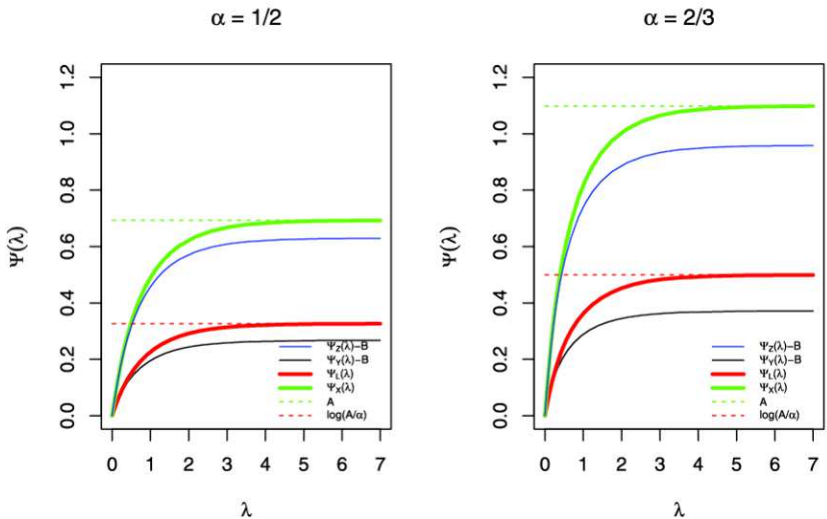}
\caption{Selection B: Comparative plot of main Bernstein functions after rescaling with $\alpha $ equal to $1/2$ and $2/3$, where $\psi _{Y}(\infty )< \psi _{L}(\infty )< \psi _{Z}(\infty )< \psi _{X}(\infty )$, knowing that $ \psi '_{L}(0)= \psi '_{Y}(0)$ and $ \psi '_{Z}(0)= \psi '_{X}(0)$}
\label{f3}
\end{figure}

\textit{Selection B.} If we choose
\begin{equation*}
\beta = \frac{1}{A}-\frac{1-\alpha }{\alpha }>0,\quad b=\frac{A\alpha
}{\alpha -A(1-\alpha )}>0,
\end{equation*}
then the mathematical expectations $ \psi '_{L}(0)= \psi '_{Y}(0)$ and
$ \psi '_{Z}(0)= \psi '_{X}(0)$, and the total masses of the L\'{e}vy
measures\index{L\'{e}vy measure} are in the following inequalities:
\begin{equation*}
\psi _{Y}(\infty )< \psi _{L}(\infty )< \psi _{Z}(\infty )< \psi _{X}(
\infty ).
\end{equation*}
Specifically,
\begin{equation*}
\log \left (2-\frac{A(1-\alpha )}{\alpha }\right )<\log \left (\frac{A}{
\alpha }\right )<\frac{\alpha (A-\alpha )}{A\alpha -A+\alpha }<A.
\end{equation*}
It is demonstrated in Figure \ref{f3}.

%r3 #&#
\begin{remark}
{All these inequalities are related to the following:
\begin{equation*} A-\alpha <A\alpha <2(A-\alpha )\end{equation*} and
\begin{equation*}A^{2}<\frac{\alpha ^{2}}{1-\alpha }, \end{equation*} where
we remark only that using (\ref{H}) we have
\begin{align*} A^{2}&=(-\log (1-\alpha ))^{2}=2\sum ^{\infty }_{n=2}|s(n,2)|\frac{\alpha ^{n}}{n!}
\\
&=2\sum ^{\infty }_{n=2}\frac{\alpha ^{n} H_{n-1}}{n}
= \alpha ^{2} +\alpha ^{3} +\frac{11}{12}\alpha ^{4} + \frac{10}{12}\alpha ^{5} +\frac{137}{180}
\alpha ^{6}\cdots \end{align*}
and
\begin{equation*}\frac{\alpha ^{2}}{1-\alpha }=\alpha ^{2}(1+\alpha +\alpha ^{2}+\cdots).\end{equation*}}

\end{remark}

%s7 #&#
\section{Conclusion}%
\label{sec7}
The Negative-Binomial process in consideration can be constructed by the
subordination\index{subordination} of a Poisson process by a Gamma process. In this way, the
process $Y(t)$ is a Poisson process subordinated by an iterated Gamma
process. In potential theory, a Gamma subordinator\index{Gamma subordinator} and an iterated Gamma
subordinator\index{Gamma subordinator} are classified as slow subordinators.

In the \textit{selection A}, the inter-arrival times of the processes
$L(t)$ and $Y(t) $ are exponentially distributed with the same parameter
$ \theta _{L}=\theta _{Y}$, but the mathematical expectation of
the Logarithmic L\'{e}vy process\index{L\'{e}vy process} in the unit time interval $ E[L(1)]$ is
less than $ E[Y(1)]$. It is the same for the processes $ X(t)$ and
$Z(t)$.

In the \textit{selection B}, the mathematical expectations of jumps
altitude for the processes $L(t) $ and $Y(t) $ are equal $
\psi '_{L}(0)= \psi '_{Y}(0)$ and $E[L(1)]=E[Y(1)] $, but the mean
number of jumps in the unit time interval of the Logarithmic process is
greater than that for the subordinated process\index{subordinated process} $Y(t)$, $\theta _{L}>
\theta _{Y} $. It is the same for the processes $ X(t)$ and $Z(t) $,
$\theta _{X}>\theta _{Z} $.

In the general setting: $ Y(t)=X(T(t))$, when the underlying process
$X(t)$ is a compound Poisson process without drift, any randomness of
$T(t)$ before it passes the level given by the first jump time of
$X(t)$ is not reflected by $Y(t)=X(T(t))$, see \cite{M}.

\begin{acknowledgement}%[title={Acknowledgments}]
The authors are very thankful to the anonymous referees for the valuable
comments on the paper.
\end{acknowledgement}

%%%%%%%%%
%%%%%%%%%\def\bisbn#1{ISBN #1}
%%%%%%%%%\def\binits#1{#1}
%%%%%%%%%\def\bauthor#1{#1}
%%%%%%%%%\def\batitle#1{#1}
%%%%%%%%%\def\bjtitle#1{#1}
%%%%%%%%%\def\bvolume#1{\textbf{#1}}
%%%%%%%%%\def\byear#1{#1}
%%%%%%%%%\def\bissue#1{#1}
%%%%%%%%%\def\bfpage#1{#1}
%%%%%%%%%\def\blpage#1{#1}
%%%%%%%%%\def\burl#1{\textsf{#1}}
%%%%%%%%%\def\doiurl#1{\textsf{#1}}
%%%%%%%%%\def\binstitute#1{#1}
%%%%%%%%%\def\binstitutionaled#1{#1}
%%%%%%%%%\def\bctitle#1{#1}
%%%%%%%%%\def\beditor#1{#1}
%%%%%%%%%\def\bpublisher#1{#1}
%%%%%%%%%\def\bbtitle#1{#1}
%%%%%%%%%\def\bedition#1{#1}
%%%%%%%%%\def\bseriesno#1{#1}
%%%%%%%%%\def\blocation#1{#1}
%%%%%%%%%\def\bsertitle#1{#1}
%%%%%%%%%\def\bsnm#1{#1}
%%%%%%%%%\def\bsuffix#1{#1}
%%%%%%%%%\def\bparticle#1{#1}
%%%%%%%%%\def\barticle#1{#1}
%%%%%%%%%\def\bconfdate#1{#1}
%%%%%%%%%\def\botherref#1{#1}
%%%%%%%%%\def\bchapter#1{#1}
%%%%%%%%%\def\bbook#1{#1}
%%%%%%%%%\def\bcomment#1{#1}
%%%%%%%%%\def\oauthor#1{#1}
%%%%%%%%%\def\endbibitem{}
%%%%%%%%%\def\bconflocation#1{#1}
%%%%%%%%%\def\arxivurl#1{\textsf{#1}}
%%%%%%%%%\csname       PreBibitemsHook\endcsname
%%%%%%%%%
%%%%%%%%%

%\begin{appendix}
%\end{appendix}

%\begin{funding}
%\gsponsor[id=,sponsor-id=]{}
%\gnumber[refid=]{}
%\end{funding}

\end{document}